\documentclass[12pt]{article}
\usepackage{stmaryrd}
\usepackage{times}
\usepackage{booktabs}

\usepackage{arydshln}
\usepackage{pifont}
\usepackage{caption}
\usepackage{mathrsfs}
\usepackage[fleqn]{amsmath}
\usepackage{amsfonts,amsthm,amssymb,mathrsfs,bbding}
\usepackage{txfonts}
\usepackage{graphics,multicol}
\usepackage{graphicx}
\usepackage{color}
\usepackage{caption}
\usepackage{indentfirst}
\usepackage{cite}
\usepackage{latexsym,bm}
\usepackage{enumerate}
\evensidemargin=\oddsidemargin\topmargin=-1.5cm

\newtheorem{theorem}{Theorem}[section]

\newtheorem{example}[theorem]{Example}
\newtheorem{remark}[theorem]{Remark}
\newtheorem{proposition}[theorem]{Proposition}
\newtheorem{corollary}[theorem]{Corollary}
\usepackage{amsmath,epsfig}

\numberwithin{equation}{section}

\begin{document}
\title{Quadratic Embedding Constants of Graph Joins%
\footnote{Supported by the National Natural Science
Foundation of China (Grant Nos. 12061074, 11971274)
and JSPS Grant-in-Aid for Scientific Research (B) 19H01789.}}
\author{
{\small Zhenzhen Lou$^{a, c}$, \ \ Nobuaki Obata$^{b,}$\footnote{Corresponding author.
\newline{\it \hspace*{5mm}Email addresses:}
xjdxlzz@163.com (Z. Z. Lou),
obata@tohoku.ac.jp (N. Obata),
huangqx@xju.edu.cn (Q. X. Huang).
}\;,\; \ \
Qiongxiang Huang$^c$}\\[2mm]
\footnotesize $^a$ College of Science, University of Shanghai for Science and Technology, Shanghai, 200093, China\\
%Department of Mathematics, East China University of Science and Technology, Shanghai, 200237, China\\
\footnotesize $^b$ Graduate School of Information Sciences,
Tohoku University, Sendai 980-8579, Japan\\
\footnotesize $^c$ College of Mathematics and Systems Science, Xinjiang University,
Urumqi, Xinjiang 830046, China}
\date{}
\maketitle
\begin{abstract}
The quadratic embedding constant (QE constant) of a graph
is a new characteristic value of a graph
defined through the distance matrix.
We derive formulae for the QE constants
of the join of two regular graphs, double graphs and
certain lexicographic product graphs.
Examples include complete bipartite graphs,
wheel graphs, friendship graphs, completely split graph,
and some graphs associated to strongly regular graphs.
\end{abstract}

\section{Introduction}\label{se-1}

In the study of spectral characteristics of a graph
the adjacency matrix and the Laplacian matrix have played
central roles, see e.g.,
\cite{BiyikogluLeydold2007,Brouwer-Haemers2012,Cvetkovic1980,VanMieghem2011}.
For further development of spectral graph theory
there is growing interest in the distance matrix
keeping a profound relation to Euclidean distance geometry.
For some recent works on distance spectrum see e.g.,
\cite{Aouchiche-Hansen2014, Lin2013,Lin2015,Liu-Xue-Guo2015}.
The quadratic embedding constant (QE constant) of a graph,
recently introduced by means of the distance matrix,
is a new characteristic value of a graph \cite{Obata-Zakiyyah2018}.
In this paper we derive formulae for the QE constants
of the join of two regular graphs, double graphs and
certain lexicographic product graphs.
Moreover, we give some computational examples related to strongly
regular graphs.
Our results support potential interest in classifying
graphs in terms of the QE constants.

Let $G=(V,E)$ be a finite connected graph
(always assumed to be simple and undirected),
where $V$ is the set of vertices and $E$ the set of edges.
Let $d(x,y)=d_G(x,y)$ be the \textit{distance}
between two vertices $x,y\in V$,
that is, the length of a shortest path connecting them,
and $D=[d(x,y)]_{x,y\in V}$ the distance matrix.
Various realizations of a graph in a Euclidean space
or in a more general metric space have been discussed
from different aspects,
see \cite{Deza-Laurent1997} and references cited therein,
and also \cite{Graham-Winkler1985,
Koolen-Shpectorov1994, Maehata2013}.
A \textit{quadratic embedding} of a graph $G=(V,E)$
in a Euclidean space $\mathbb{R}^N$ is a map
$\varphi:V\rightarrow \mathbb{R}^N$ satisfying
\[
\|\varphi(x)-\varphi(y)\|^2
=d(x,y),
\qquad
x,y\in V,
\]
where the left-hand side is the square of the Euclidean distance
between two points $\varphi(x)$ and $\varphi(y)$.
A graph $G$ is called
\textit{of QE class} or \textit{of non-QE class} according as
$G$ admits a quadratic embedding or not.

The concept of quadratic embedding traces back to the early works of
Schoenberg \cite{Schoenberg1935,Schoenberg1938}
(see also \cite{Young-Householder1938})
and has been studied considerably
along with Euclidean distance geometry
\cite{Alfakih2018, Balaji-Bapat2007,
Jaklic-Modic2013,Jaklic-Modic2014, Liberti-etal2014}.
Schoenberg's theorem says that
a finite connected graph $G$ is of QE class
if and only if the distance matrix  $D$ is
conditionally negative definite, namely,
$\langle f, Df\rangle\le0$ for all $f\in C(V)$ with
$\langle \mathbf{1},f\rangle=0$,
where $C(V)$ is the space of $\mathbb{R}$-valued functions on $V$,
$\mathbf{1}\in C(V)$ the constant function taking value one,
and $\langle\cdot,\cdot\rangle$
the canonical inner product on $C(V)$.
Moreover, it is noteworthy that
$D$ is conditionally negative definite
if and only if the $Q$-matrix $Q=[q^{d(x,y)}]{x,y\in V}$ is
positive definite (allowing zero eigenvalues)
for all $0\le q\le1$.
The last condition is essential for noncommutative
harmonic analysis on free groups \cite{Bozejko89} and
$q$-deformed spectral analysis of growing graphs
\cite{Hora-Obata2007}.

A new quantitative approach was initiated
in the recent paper \cite{Obata-Zakiyyah2018}.
For a finite connected graph $G=(V,E)$ with $|V|\ge2$
the \textit{quadratic embedding constant} (QE constant for short)
is defined by
\begin{equation}\label{1eqn:def of QEC}
\mathrm{QEC}(G)
=\max\{\langle f,Df\rangle\,;\,  f\in C(V), \,
\langle f,f\rangle=1, \,
\langle \mathbf{1},f\rangle=0\}.
\end{equation}
By definition, a graph on at least two vertices
admits a quadratic embedding if and only if $\mathrm{QEC}(G)\le0$.
Thus, the QE constant gives rise to a criterion for
graphs to be of QE class or not.
Moreover, the QE constant is interesting for itself
as a new numerical invariant of graphs.
In the previous papers
\cite{MO-2018,Obata2017,Obata-Zakiyyah2018}
we obtained explicit values of the QE constants
of particular graphs
and their estimates in relation to graph operations.
More concrete examples of QE constants have been
obtained in the recent papers
\cite{Irawan-Sugeng2021,M-2022,Purwaningsih-Sugeng2021}.
We are also interested in classifying
the finite connected graphs in terms of the QE constants,
see e.g., \cite{BO-2019}, where
we started an attempt to characterize graphs along with
the increasing sequence of the QE constants of paths.

Another interesting aspect of the QE constant
is found in a relation to the minimal eigenvalue of
the adjacency matrix.
In fact, for a connected strongly regular graph $G$
we will derive the simple relation:
\[
\mathrm{QEC}(G)=-2-\lambda_{\min}(G),
\]
see Propositions \ref{02prop:QEC(regular and diam<=2}
and \ref{05prop:QEC(SRG)}.
Therefore, for a connected strongly regular graph $G$
the condition
$\mathrm{QEC}(G)\le0$ is equivalent to $\lambda_{\min}(G)\ge-2$.
On the other hand,
graphs with smallest eigenvalue at least $-2$ have been
extensively studied and strongly regular graphs
with such property are classified
\cite{Brouwer-Cohen-Neumaier1989,
Cameron1976,Cheng-Greaves-Koolen2018}.
In this aspect too the QE constant is expected
to provide an interesting characteristic value of a graph.

Finally, as is expected from the definition \eqref{1eqn:def of QEC}
the QE constant is closely related to the distance spectrum,
for the comprehensive account
see \cite{Aouchiche-Hansen2014} and references cited therein.
The eigenvalues of the distance matrix $D$ are arranged as
\[
\delta_1(G)>\delta_2(G)\ge\delta_3(G)\ge\dotsb\ge \delta_n(G)\,,
\quad n=|V|,
\]
where the first strict inequality,
due to the Perron-Frobenius theorem,
means that the maximal eigenvalue of $D$ is simple.
In this line a remarkable relation to the QE constant is
shown by the inequalities:
\[
\delta_2(G)\le \mathrm{QEC}(G)<\delta_1(G),
\]
the proof of which is by the standard min-max theorem.
It is an interesting open question to
characterize graphs with $\delta_2(G)=\mathrm{QEC}(G)$.

This paper is organized as follows.
In Section 2 we review some methods for
calculating the QE constants as well as their basic properties.
In Section 3 we derive a formula for the QE constant
of the join $G_1+G_2$ of regular graphs $G_1$ and $G_2$ (Theorem \ref{03thm:QEC(G_1+G_2)})
and show typical examples.
In Section 4 we discuss
the double graph $\mathrm{Double}(G)$ and
lexicographic product graph $G\triangleright K_2$,
which are subgraphs of the graph join $G+G$.
We see that the formulae for their QE constants
(Theorems \ref{4thm:QEC(Double)} and \ref{4thm:QEC(lexicographic)})
are compatible to the ones for distance spectra
\cite{Indulal-Gutman2008}.
In Section 5 we examine some strongly regular graphs and their
graph joins,
and construct families of graphs which are not of QE class
(Theorem \ref{5-thm-1}).

%%%%%%%%%%%%%%%%%%%%%%%%%%%
\section{Preliminaries}

\subsection{Some Basic Properties of QE Constants}

As first examples, for the complete graphs
and complete bipartite graphs we have
\begin{align}
&\mathrm{QEC}(K_n)=-1,
\qquad n\ge2,
\label{02eqn:QE(K_n)}\\
&\mathrm{QEC}(K_{m,n})=\frac{2(mn-m-n)}{m+n}\,,
\qquad m\ge1, \quad n\ge1.
\label{02eqn:QE(K_mn)}
\end{align}
For cycles we have
\begin{equation}
\mathrm{QEC}(C_n)=
\begin{cases}
-\left(4\cos^2\dfrac{\pi}{n}\right)^{-1},
& \text{if $n\ge3$ is odd}, \\
0, & \text{if $n\ge4$ is even}.
\end{cases}
\end{equation}
For paths $\mathrm{QEC}(P_n)$ can be seen \cite{M-2022} and a  sharp estimate can be seen \cite{MO-2018}.
More examples of QE constants can be found in
\cite{BO-2019,Obata2017,Obata-Zakiyyah2018}.

Let $G=(V,E)$ be a finite connected graph
and $H=(W,F)$ a connected subgraph with $|W|\ge2$.
Let $D_G$ and $D_H$ be the distance matrices of $G$ and $H$,
respectively.
If $H$ is isometrically embedded in $G$, that is,
\[
d_H(x,y)=d_G(x,y),
\qquad x,y\in W,
\]
the distance matrix $D_H$ becomes a submatrix of $D_G$
and by definition we have
\[
\mathrm{QEC}(H)\le \mathrm{QEC}(G).
\]
In particular, every
finite connected graph $G=(V,E)$ with $|V|\ge2$ fulfills
\[
-1\le \mathrm{QEC}(G),
\]
since $G$ contains $K_2$ isometrically and
$\mathrm{QEC}(K_2)=-1$.
Moreover, if $G=(V,E)$ is not complete,
it contains $P_3$ isometrically and we have
\[
\mathrm{QEC}(P_3)=-\frac23\le \mathrm{QEC}(G).
\]
Hence $\mathrm{QEC}(G)=-1$ if and only if $G$ is a complete graph.

\subsection{Formulae for QE Constants}
\label{subsec:Formulae for QE Constants}

We first recall a general formula for
the QE constant \eqref{1eqn:def of QEC} for
a graph $G=(V,E)$ with $V=\{1,2,\dots,n\}$, $n\ge3$.
Following the method of Lagrange's multipliers we consider
\begin{equation}\label{01eqn:def of F}
F(f,\lambda,\mu)
=\langle f,Df\rangle
-\lambda(\langle f,f\rangle-1)
-\mu \langle\mathbf{1},f\rangle.
\end{equation}
Identifying $f\in C(V)$ with $[f_i]\in \mathbb{R}^n$ in
such a way that $f(i)=f_i$, we regard
$F(f,\lambda,\mu)$ as a function of
$(f=[f_i],\lambda,\mu)\in\mathbb{R}^n\times \mathbb{R}\times\mathbb{R}$.
Let $\mathcal{S}(D)$ be the set of all stationary points
of $F(f,\lambda,\mu)$,
that is, the set of $(f,\lambda,\mu)
\in\mathbb{R}^n\times\mathbb{R}\times\mathbb{R}$ such that
\begin{equation}\label{01eqn:stationary points}
\frac{\partial F}{\partial f_1}
=\dots
=\frac{\partial F}{\partial f_n}
=\frac{\partial F}{\partial \lambda}
=\frac{\partial F}{\partial \mu}
=0.
\end{equation}
The above equations are written explicitly as follows:
\begin{equation}\label{01eqn:LM1}
\frac{\partial F}{\partial f_i}
=2Df(i)-2\lambda f(i)-\mu=0,
\quad i=1,2,\dots,n,
\end{equation}
or equivalently,
\begin{equation}\label{01eqn:LM1-1}
(D-\lambda)f=\frac{\mu}{2}\,\bm{1},
\end{equation}
and
\begin{align}
\langle f,f\rangle=1,
\label{01eqn:LM2}\\
\langle \mathbf{1},f \rangle=0.
\label{01eqn:LM3}
\end{align}
We then see that $\mathcal{S}(D)$ is the set of
$(f,\lambda,\mu)\in\mathbb{R}^n\times\mathbb{R}\times\mathbb{R}$
satisfying \eqref{01eqn:LM1-1}--\eqref{01eqn:LM3}.

\begin{proposition}[\cite{Obata-Zakiyyah2018}]
\label{02prop:QEC}
Let $G=(V,E)$ be a finite connected graph
on $n=|V|\ge3$ vertices.
Then we have
\[
\mathrm{QEC}(G)=\max\Lambda(D),
\]
where $\Lambda(D)$ is the projection of
$\mathcal{S}(D)$ onto the $\lambda$-axis,
that is, the set of $\lambda\in\mathbb{R}$ such that
$(f,\lambda,\mu)\in\mathcal{S}(D)$ for some
$f\in\mathbb{R}^n$ and $\mu\in\mathbb{R}$.
\end{proposition}

The \textit{diameter} of
a finite connected graph $G=(V,E)$ is defined by
\[
\mathrm{diam}(G)=
\max\{d(x,y)\,;\, x,y\in V\}.
\]
Recall that the adjacency matrix $A=[a(x,y)]_{x,y\in V}$ of $G$ is
defined by setting $a(x,y)=1$ if $x$ and $y$ are adjacent
and $a(x,y)=0$ otherwise.
If $1\leq \mathrm{diam}(G)\leq2$,
then $|V|\ge2$ and we have
\begin{equation}\label{1eqn:D}
D=2J-2I-A,
\end{equation}
where $J$ is the matrix whose entries are all one,
and $I$ the identity matrix.
(The symbols $J$ and $I$ are used without
specifying their sizes.)

\begin{proposition}[\cite{Obata2017}]\label{2-eq-1}
For a finite connected graph $G$ with $1\leq \mathrm{diam}(G)\leq2$
we have
\begin{equation}\label{2eqn:in Prop2.2}
\mathrm{QEC}(G)
=-2-\min\left\{\langle f,Af\rangle \,;\,
 f\in C(V), \,
 \langle f,f\rangle=1, \,
 \langle \mathbf{1},f\rangle=0
 \right\}.
\end{equation}
\end{proposition}

\begin{proposition}\label{02prop:QEC(regular and diam<=2}
For a finite connected regular graph $G$
with $1\leq \mathrm{diam}(G)\leq2$ we have
\begin{equation}\label{2eqn:in Prop2.3}
\mathrm{QEC}(G)
=-2-\lambda_{\mathrm{min}}(G),
\end{equation}
where $\lambda_{\mathrm{min}}(G)=\min\mathrm{ev}(A)$ is
the minimal eigenvalue of the adjacency matrix $A$.
\end{proposition}

\begin{proof}
If $G$ is a regular connected graph of degree $r\ge0$,
we have $r\in\mathrm{ev}(A)$ and $A\bm{1}=r\bm{1}$.
Moreover, it is known that $r$ is the maximal eigenvalue of $A$
and is simple.
Then, \eqref{2eqn:in Prop2.3} follows from
\eqref{2eqn:in Prop2.2} immediately.
\end{proof}

Obviously, $\mathrm{diam}(G)=1$ if and only if
$G$ is a complete graph $G=K_n$ with $n\ge2$.
Then $\mathrm{QEC}(K_n)=-1$ follows also from
Proposition \ref{02prop:QEC(regular and diam<=2}.
Graphs with $\mathrm{diam}(G)=2$ are interesting
from several points of view, e.g., \cite{Brouwer-Cohen-Neumaier1989}.
We will discuss in Section \ref{sec:Strongly Regular Graphs}
some topics related to strongly regular graphs.

\subsection{Graph Join}

For $i=1,2$ let $G_i=(V_i,E_i)$ be
a finite graph with $V_1\cap V_2=\emptyset$,
and $A_i$ its adjacency matrix.
Set
\begin{align*}
\Tilde{V}&=V_1\cup V_2, \\
\Tilde{E}&=E_1\cup E_2 \cup\{\{x,y\}\,;\,x\in V_1, y\in V_2\}.
\end{align*}
Then $\Tilde{G}=(\Tilde{V},\Tilde{E})$ becomes
a graph, which is called the
\textit{join} of $G_1$ and $G_2$,
and is denoted by
\[
\Tilde{G}=G_1+G_2\,.
\]
With the natural arrangement of rows and columns
the adjacency matrix of $\Tilde{G}=G_1+G_2$ is written
in a block-matrix form:
\[
\Tilde{A}=\begin{bmatrix}
  A_1 & J \\
   J  & A_2
\end{bmatrix}.
\]
Note the graph join $G_1+G_2$ becomes connected
even when $G_1$ and $G_2$ are not connected.
Moreover, we have $1\leq \mathrm{diam}(G_1+G_2)\leq2$.
We then see from \eqref{1eqn:D} that
the distance matrix of $\Tilde{G}=G_1+G_2$ is given by
\begin{equation}\label{2eqn:distance matrix of join}
\Tilde{D}=2J-2I-\Tilde{A}
=\begin{bmatrix}
  2J-2I-A_1 & J \\
   J  & 2J-2I-A_2
\end{bmatrix}.
\end{equation}

For the QE constant of $\Tilde{G}=G_1+G_2$ we
need to investigate $\mathcal{S}(\Tilde{D})$,
the set of stationary points of
\[
F(f,\lambda,\mu)
=\langle f,\Tilde{D} f\rangle
-\lambda(\langle f,f\rangle-1)
-\mu\langle \bm{1},f\rangle,
\qquad
f\in C(\Tilde{V}),\,\,
\lambda\in\mathbb{R},\,\,
\mu\in\mathbb{R}.
\]
According to the block-matrix notation in
\eqref{2eqn:distance matrix of join} we write
\[
f=\begin{bmatrix} g \\ h \end{bmatrix},
\qquad g\in C(V_1),
\quad h \in C(V_2).
\]
Then after simple algebra, we see that
$\mathcal{S}(\Tilde{D})$ is the set of
$(g,h,\lambda,\mu)\in C(V_1)\times C(V_2)\times
\mathbb{R}\times\mathbb{R}$ satisfying
\begin{align}
&2(A_1+\lambda+2)g
-(2\langle\bm{1},g\rangle-\mu)\bm{1}=0,
\label{2eqn:Join1}\\
&2(A_2+\lambda+2)h
-(2\langle\bm{1},h\rangle-\mu)\bm{1}=0,
\label{2eqn:Join2}\\
&\langle g,g\rangle+\langle h,h\rangle=1,
\label{2eqn:Join3}\\
&\langle \bm{1},g\rangle+\langle \bm{1},h\rangle=0.
\label{2eqn:Join4}
\end{align}
Then as an immediate consequence of Proposition \ref{02prop:QEC}
we obtain

\begin{proposition}\label{02prop:general formula QEC(G_1+G_2)}
Let $G_1=(V_1,E_1)$ and $G_2=(V_2,E_2)$ be finite
(not necessarily connected) graphs with $V_1\cap V_2=\emptyset$.
Let $\mathcal{S}(\Tilde{D})$ be the same as above
and $\Lambda(\Tilde{D})$ the projection
of $\mathcal{S}(\Tilde{D})$ onto the $\lambda$-axis,
namely, the set of $\lambda\in\mathbb{R}$ such that
$(g,h,\lambda,\mu)\in \mathcal{S}(\Tilde{D})$
for some $g\in C(V_1)$, $h\in C(V_2)$ and $\mu\in\mathbb{R}$.
Then we have
\begin{equation}\label{02eqn:QEC(join of graphs)}
\mathrm{QEC}(G_1+G_2)
=\max\Lambda(\Tilde{D}),
\end{equation}
\end{proposition}

%%%%%%%%%%%%%%%%%%%%%%%%%%%%%%%%%%%%%%%%%%%%%%%%%%%%%
\section{Join of Regular Graphs}

The main purpose of this section is to derive a formula
of the QE constant of the join of two regular graphs.
The result is stated in the following

\begin{theorem}\label{03thm:QEC(G_1+G_2)}
For $i=1,2$ let $G_i=(V_i,E_i)$ be
a $r_i$-regular graph on $n_i=|V_i|$ vertices
with $V_1\cap V_2=\emptyset$,
where $r_i\ge0$ and $n_i\ge1$.
Then we have
\begin{equation}\label{03eqn:QEC(G_1+G_2)}
\mathrm{QEC}(G_1+G_2)=-2+
\max\left\{
-\lambda_{\min}(G_1),\,
-\lambda_{\min}(G_2),\,
\frac{2n_1n_2-r_1n_2-r_2n_1}{n_1+n_2}\right\},
\end{equation}
where $\lambda_{\min}(G_i)$ is the minimal eigenvalue of
the adjacency matrix of $G_i$.
\end{theorem}

\begin{proof}
Note that $r_i=0$ is allowed.
In that case $G_i$ is an empty graph on $n_i$ vertices.
For \eqref{03eqn:QEC(G_1+G_2)}
we will describe $\mathcal{S}(\Tilde{D})$ explicitly
and apply Proposition \ref{02prop:general formula QEC(G_1+G_2)}.

Since $G_1$ is $r_1$-regular, we have $A_1\bm{1}=r_1\bm{1}$.
Then, taking the inner product of \eqref{2eqn:Join1} with $\bm{1}$,
we obtain
\begin{equation}\label{03eqn:1}
(\lambda+2-n_1+r_1)\langle\bm{1},g\rangle=-\frac{\mu}{2}\,n_1\,.
\end{equation}
Similarly, from \eqref{2eqn:Join2} we obtain
\[
(\lambda+2-n_2+r_2)\langle\bm{1},h\rangle=-\frac{\mu}{2}\,n_2\,,
\]
and using \eqref{2eqn:Join4} we have
\begin{equation}\label{03eqn:3}
(\lambda+2-n_2+r_2)\langle\bm{1},g\rangle=\frac{\mu}{2}\,n_2\,.
\end{equation}
It then follows from \eqref{03eqn:1} and \eqref{03eqn:3} that
\[
\{(n_1+n_2)(\lambda+2)-2n_1n_2+r_1n_2+r_2n_1\}\langle\bm{1},g\rangle=0.
\]
Accordingly, we consider two cases.

(Case 1)\, $(n_1+n_2)(\lambda+2)+n_1r_2+n_2r_1-2n_1n_2=0$.
This happens when $\lambda=\lambda^*$, where
\begin{equation}\label{02eqn:candidate}
\lambda^*=\frac{2n_1n_2-r_1n_2-r_2n_1}{n_1+n_2}-2.
\end{equation}
We set
\[
g^*=\sqrt{\frac{n_2}{n_1(n_1+n_2)}}\, \bm{1},
\qquad
h^*=-\sqrt{\frac{n_1}{n_2(n_1+n_2)}}\, \bm{1},
\]
and
\[
\mu^*=\frac{2(n_1-r_1-n_2+r_2)}{n_1+n_2}
    \sqrt{\frac{n_1n_2}{n_1+n_2}}\,.
\]
Then, simple calculation shows that
\eqref{2eqn:Join1}--\eqref{2eqn:Join4} are satisfied by
$(g^*,h^*,\lambda^*,\mu^*)$.
Hence $\lambda^*\in\Lambda(\Tilde{D})$.

(Case 2)\, $\lambda\neq \lambda^*$ and $\langle\bm{1},g\rangle=0$.
By \eqref{2eqn:Join4} and \eqref{03eqn:1} we obtain
$\langle\bm{1},h\rangle=\mu=0$.
Moreover, \eqref{2eqn:Join1} and \eqref{2eqn:Join2} become
\begin{align}
&(A_1+\lambda+2)g=0,
\label{2eqn:Join1-1}\\
&(A_2+\lambda+2)h=0,
\label{2eqn:Join2-1}
\end{align}
respectively.
Thus, our task is to find $(g,h,\lambda,0)$ satisfying
\eqref{2eqn:Join1-1}, \eqref{2eqn:Join2-1},
\eqref{2eqn:Join3} and
\begin{equation}\label{03eqn:Join4-2}
\langle\bm{1},g\rangle=\langle\bm{1},h\rangle=0.
\end{equation}
It follows from \eqref{2eqn:Join1-1} and
\eqref{2eqn:Join2-1} that
$-\lambda-2 \in\mathrm{ev}(A_1)\cup\mathrm{ev}(A_2)$.
In fact, if otherwise we have $g=0$ and $h=0$ which do not
fulfill \eqref{2eqn:Join3}.
Hence
\begin{equation}\label{3eqn:subset1}
\Lambda(\Tilde{D})\backslash\{\lambda^*\}
\subset
\{-\alpha-2\,;\, \alpha\in\mathrm{ev}(A_1)\cup \mathrm{ev}(A_2)\}.
\end{equation}

Now recall that $r_1\in \mathrm{ev}(A_1)$ is simple
and $A_1\bm{1}=r_1\bm{1}$.
Let $\alpha\in\mathrm{ev}(A_1)\backslash\{r_1\}$
and take an eigenvector $g$ such that $\langle g,g\rangle=1$.
Since the eigenspaces with distinct eigenvalues
are mutually orthogonal, we have $\langle \bm{1},g\rangle=0$.
Then $(g,h=0,-\alpha-2,0)$ satisfies
\eqref{2eqn:Join1-1}, \eqref{2eqn:Join2-1},
\eqref{2eqn:Join3} and \eqref{03eqn:Join4-2}.
Hence $-\alpha-2\in\Lambda(\Tilde{D})$.
Similarly,
if $\alpha\in\mathrm{ev}(A_2)\backslash\{r_2\}$, we have
$-\alpha-2\in\Lambda(\Tilde{D})$.
Thus,
\begin{equation}\label{3eqn:subset2}
\{-\alpha-2\,;\, \alpha\in\mathrm{ev}(A_1)\backslash\{r_1\}\}
\cup \{-\alpha-2\,;\, \alpha\in \mathrm{ev}(A_2)\backslash\{r_2\}\}
\subset \Lambda(\Tilde{D})
\end{equation}

Since $r_1$ is the largest eigenvalue of $A_1$,
\[
\max\{-\alpha-2\,;\, \alpha\in\mathrm{ev}(A_1)\}
=\max\{-\alpha-2\,;\, \alpha\in\mathrm{ev}(A_1)\backslash\{r_1\}\}
=-\lambda_{\min}(G_1)-2.
\]
A similar relation holds for $r_2$ and $A_2$.
Consequently, we see from \eqref{3eqn:subset1} and
\eqref{3eqn:subset2} that
\[
\max\Lambda(\Tilde{D})
=\max\{-\lambda_{\min}(G_1)-2, \,
-\lambda_{\min}(G_2)-2, \,\lambda^*\},
\]
which proves \eqref{03eqn:QEC(G_1+G_2)}.
\end{proof}

\begin{corollary}\label{03cor:bound QEC(G_1+G_2)}
Notations and assumptions being the same as in
Theorem \ref{03thm:QEC(G_1+G_2)}, we have
\[
\mathrm{QEC}(G_1+G_2)\ge
\max\{-2-\lambda_{\min}(G_1),\,\,
-2-\lambda_{\min}(G_2)\}.
\]
\end{corollary}

\begin{example}[complete bipartite graphs \cite{Obata-Zakiyyah2018}]
\label{03ex:complete bipartite graphs}
\normalfont
Let $m\ge1$ and $n\ge1$.
The \textit{complete bipartite graph} $K_{m,n}$ is the graph join
$K_{m,n}=\bar{K}_m+\bar{K}_n$, where
$\bar{K}_m$ and $\bar{K}_m$ the empty graphs
on $m$ and $n$ vertices, respectively.
Obviously,
\begin{equation}\label{03eqn:lambda_min(bar K_m)}
\lambda_{\min}(\bar{K}_m)=0,
\qquad m\ge1.
\end{equation}
It then follows from Theorem \ref{03thm:QEC(G_1+G_2)} that
\[
\mathrm{QEC}(G_1+G_2)=-2+
\max\left\{0,0,\frac{2mn}{m+n}\right\}
=-2+\frac{2mn}{m+n}.
\]
Hence
\[
\mathrm{QEC}(K_{m,n})
=\frac{2(mn-m-n)}{m+n}\,,
\qquad m\ge1,\quad n\ge1.
\]
\end{example}

\begin{example}[complete split graph]\label{exa-4}
\normalfont
Let $m\ge1$ and $n\ge1$.
The graph join $K_n + \bar{K}_m$
is called the \textit{complete split graph}.
Note that
\begin{equation}\label{03eqn:lambda_min(K_n)}
\lambda_{\min}(K_n)
=\begin{cases}
  0, & \text{for $n=1$,} \\
  -1, & \text{for $n\ge2$.}
  \end{cases}
\end{equation}
For $n=1$ we have
$K_n+\bar{K}_m=K_1+\bar{K}_m=\bar{K}_1+\bar{K}_m=K_{1,m}$,
which is a special case of
Example \ref{03ex:complete bipartite graphs}.
Assume that $n\ge2$ and $m\ge1$.
Then by Theorem \ref{03thm:QEC(G_1+G_2)} we obtain
\begin{align}
\mathrm{QEC}(K_n +\bar{K}_m)
&=-2+\max\left\{1,0,\frac{2mn-(n-1)m}{m+n} \right\}
\nonumber\\
&=-2+\frac{mn+m}{m+n}=\frac{mn-m-2n}{m+n}.
\label{03eqn:inEX3.4}
\end{align}
In particular,
the tri-partite graph $K_{1,1,m}$ being the join of
$K_2$ and $\bar{K}_m$,
we set $n=2$ in \eqref{03eqn:inEX3.4} to get
\[
\mathrm{QEC}(K_{1,1,m})
=\mathrm{QEC}(K_2+\bar{K}_m)=\frac{m-4}{m+2},
\]
see also \cite{Obata-Zakiyyah2018}.
\end{example}

\begin{example}[friendship graphs]\label{exa-2}
\normalfont
Let $n\ge1$. The \textit{friendship graph} $F_n$ on $2n+1$ vertices
is the graph join $F_n=nK_2+ K_1$,
where $nK_2$ is the disjoint union of $n$ copies of $K_2$.
Note that $nK_2$ is a $1$-regular graph on $2n$ vertices.
Since $\lambda_{\min}(nK_2)=\lambda_{\min}(K_2)=-1$
and $\lambda_{\min}(K_1)=0$,
by Theorem \ref{03thm:QEC(G_1+G_2)} we obtain
\[
\mathrm{QEC}(F_n)
=-2+\max\left\{1,\, 0,\, \frac{4n-1}{2n+1}\right\}
=-2+\frac{4n-1}{2n+1}=\frac{-3}{2n+1}.
\]
\end{example}

\begin{example}\label{exa-1}
\normalfont
Let $n\ge3$ and $m\ge1$.
We consider the graph joins $C_n+K_m$ and $C_n+\bar{K}_m$.
Recall first that
\[
\lambda_{\min}(C_n)
=\min\left\{
 2\cos\frac{2\pi j}{n}\,;\, 1\le j\le n-1\right\}
=\begin{cases}
 -2+4\sin^2\dfrac{\pi}{2n}\,, & \text{if $n$ is odd,}\\
 -2, & \text{if $n$ is even,}
\end{cases}
\]
of which verification is elementary,
see e.g., \cite{Brouwer-Haemers2012}.
Since $\lambda_{\min}(C_n)\le\lambda_{\min}(K_n)$ for
all $n\ge3$ and $m\ge1$, we obtain
\begin{align}
\mathrm{QEC}(C_n +K_m)
&=-2+\max\left\{-\lambda_{\min}(C_n), \,\,\frac{mn-2m+n}{m+n} \right\}
\nonumber\\
&=\max\left\{-2-\lambda_{\min}(C_n),\,\, \frac{mn-4m-n}{m+n} \right\}.
\label{3-eq-5}
\end{align}
Case of $n$ being even in \eqref{3-eq-5} is simple.
In fact, after simple algebra
we obtain
\[
\mathrm{QEC}(C_{2n} +K_m)
=\begin{cases}
0, & \text{if $mn-2m-n\le0$}, \\[3pt]
\dfrac{2mn-4m-2n}{m+2n}\,, & \text{otherwise},
\end{cases}
\]
where $m\ge1$ and $n\ge2$.
In case of $n$ being odd in \eqref{3-eq-5}, we only mention
the following formula:
\[
\mathrm{QEC}(C_{2n-1}+K_m)
=\max\left\{-4\sin^2\frac{\pi}{2(2n-1)}\,,\,\,
 \frac{2mn-5m-2n+1}{m+2n-1} \right\},
\]
where $n\ge2$ and $m\ge1$.
For $n\ge3$ the
\textit{wheel graph} $W_n$ on $n+1$ vertices is the graph join
$W_n=C_n+K_1$.
Then after simple observation we obtain
\[
\mathrm{QEC}(W_n)
=\mathrm{QEC}(C_n+K_1)
=\begin{cases}
-4\sin^2\dfrac{\pi}{2n} & \text{if $n$ is odd,}\\[3pt]
0& \text{if $n$ is even.}
\end{cases}
\]
The above result was obtained in
\cite{Obata2017} by different calculation.
In a similar and, in fact, simpler fashion we obtain
\[
\mathrm{QEC}(C_n+\bar{K}_m)
=\frac{2mn-4m-2n}{m+n}\,,
\qquad n\ge3, \quad m\ge2.
\]
In particular,
$\mathrm{QEC}(C_n+\bar{K}_m)>0$ except
$(m,n)=(2,3),(3,3),(2,4)$, and
\[
\mathrm{QEC}(C_3+\bar{K}_2)=-\frac25\,,
\quad
\mathrm{QEC}(C_3+\bar{K}_3)
=\mathrm{QEC}(C_4+\bar{K}_2)
=0.
\]
In fact, ${QEC}(C_3+\bar{K}_2)=-\frac25$ was first presented as No. 20 in the list of $n=5$ of \cite{Obata-Zakiyyah2018}.
\end{example}

\begin{remark}\normalfont
Let $G$ be a graph and $A$ be the adjacency matrix.
Following Cvetkovi\'{c} \cite{Cvetkovic1978}
an eigenvalue of $A$ is called a \textit{main eigenvalue}
if it has an eigenvector which is not orthogonal to $\bm{1}$.
Then a graph $G$ has exactly one main eigenvalue
if and only if $G$ is regular.
In that case this main eigenvalue is equal to the regular degree.
It seems interesting to study generalization of
Theorem \ref{03thm:QEC(G_1+G_2)} along the idea of
main and non-main eigenvalues.
\end{remark}

%%%%%%%%%%%%%%%%%%%%%%%%%%%%%%%%%%%%%%%%%%%%%%%%%%%%%%%%
\section{Two Subgraphs of Graph Join}

\subsection{Double Graphs $\mathrm{Double}(G)$}

Let $G=(V,E)$ be a finite graph (not necessarily connected).
We set
\begin{align*}
\Tilde{V}&=V\times\{0,1\}=\{(x,i)\,;\, x\in V, i\in\{0,1\}\},\\
\Tilde{E}&=\{
\{(x,i),(y,j)\}\,;\,(x,i)\in \Tilde{V},\,\,(y,j)\in \Tilde{V},\,\,
\{x,y\}\in E\}.
\end{align*}
Then $(\Tilde{V},\Tilde{E})$ becomes a graph,
which is called the \textit{double graph} of $G=(V,E)$
and is denoted by $\mathrm{Double\,}(G)$.
Clearly, $\mathrm{Double\,}(G)$ is a subgraph of $G+G$.
For a review on the double graphs we refer to
\cite{Munarini-etal2008}.

The double graph is understood in a slightly informal manner
as follows.
Divide $\Tilde{V}$ into two parts:
\[
\Tilde{V}=V_0\cup V_1\,,
\qquad
V_i=\{(x,i)\,;\, x\in V\},
\quad i=0,1.
\]
Then the induced subgraph spanned by $V_i$
is isomorphic to $G$.
Namely, the double graph $\mathrm{Double\,}(G)$ is
a graph obtained from the union of two copies of $G$
by adding edges between them.
The new edges appear between $(x,0)$ and $(y,1)$ if and only if
$x\in V$ and $y\in V$ are adjacent.

\begin{example}\normalfont
As is seen in Figure \ref{Fig:double graphs} we have
\[
\mathrm{Double\,}(K_2)=C_4\,,
\qquad
\mathrm{Double\,}(K_3)=K_6\backslash 3K_2\,,
\qquad
\mathrm{Double\,}(P_3)=K_{2,4}\,.
\]
In fact, for the complete graph $K_n$ we have
\[
\mathrm{Double\,}(K_n)=K_{2n}\backslash nK_2\,,
\]
where $nK_2$ stands for the disjoint union of $n$ edges.
\end{example}

\begin{figure}[hbt]
\begin{center}
\includegraphics[width=60pt]{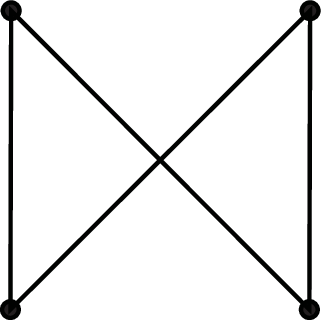}
\qquad\quad
\includegraphics[width=90pt]{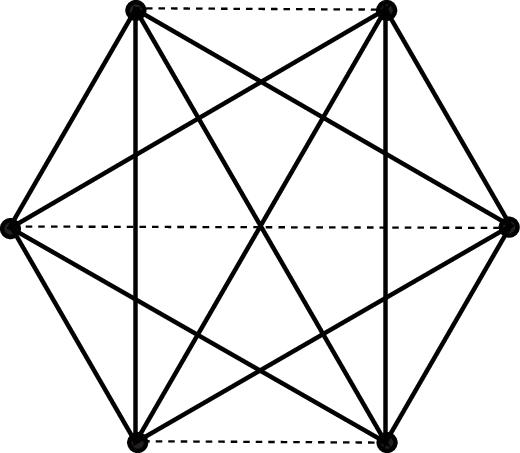}
\qquad\quad
\includegraphics[width=60pt]{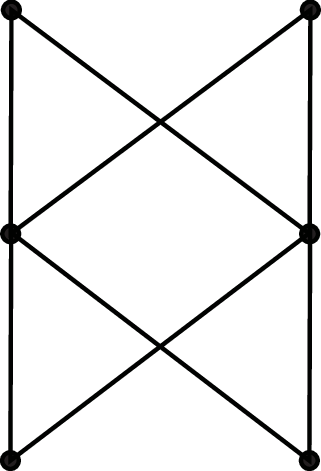}
\end{center}
\caption{$\mathrm{Double\,}(K_2)$,\,
$\mathrm{Double\,}(K_3)$ and $\mathrm{Double\,}(P_3)$}
\label{Fig:double graphs}
\end{figure}

We are interested in the QE constants of double graphs.
As is easily verified, for a graph $G=(V,E)$ with
$n=|V|\ge2$ the double graph $\mathrm{Double\,}(G)$ is
connected if and only if $G$ is connected.

\begin{theorem}\label{4thm:QEC(Double)}
Let $G=(V,E)$ be a finite connected graph with $n=|V|\ge2$.
Then we have
\begin{equation}\label{4eqn:QEC(Double)}
\mathrm{QEC}(\mathrm{Double\,}(G))
=2\mathrm{QEC}(G)+2.
\end{equation}
\end{theorem}

\begin{proof}
Let $D$ be the distance matrix of $G$.
Then, with the natural arrangement of rows and columns
the distance matrix $\Tilde{D}$ of the double graph
$\mathrm{Double\,}(G)$ is written in the form:
\begin{equation}\label{05eqn:Tilde D}
\Tilde{D}=
\begin{bmatrix}
D & D+2I \\
D+2I & D
\end{bmatrix}.
\end{equation}
According to the general argument in Subsection
\ref{subsec:Formulae for QE Constants}
we consider
\[
F(g,h,\lambda,\mu)
=\left\langle \begin{bmatrix} g \\ h \end{bmatrix},
 \Tilde{D} \begin{bmatrix} g \\ h \end{bmatrix}\right\rangle
-\lambda(\langle g,g\rangle+\langle h,h\rangle-1)
-\mu(\langle\bm{1},g\rangle+\langle\bm{1},h\rangle),
\]
where $g\in C(V_0)$, $h\in C(V_1)$, $\lambda\in\mathbb{R}$
and $\mu\in \mathbb{R}$.
Let $\mathcal{S}(\Tilde{D})$ be the set of all stationary
points of $F(g,h,\lambda,\mu)$.
By Proposition \ref{02prop:QEC} we have
\begin{equation}\label{4eqn:QEC formula}
\mathrm{QEC}(\mathrm{Double\,}(G))
=\max\Lambda(\Tilde{D}),
\end{equation}
where $\Lambda(\Tilde{D})$ is the projection
of $\mathcal{S}(\Tilde{D})$ onto the $\lambda$-axis,
namely, the set of $\lambda\in\mathbb{R}$ such that
$(g,h,\lambda,\mu)\in \mathcal{S}(\Tilde{D})$
for some $g\in C(V_0)$, $h\in C(V_1)$ and $\mu\in\mathbb{R}$.

After simple calculation we see that
$\mathcal{S}(\Tilde{D})$ is the set of solutions to
the following system of equations:
\begin{align}
&2Dh+2Dg+4g-2\lambda h-\mu\bm{1}=0,
\label{04eqn:D1}\\
&2Dh+2Dg+4h-2\lambda g-\mu\bm{1}=0,
\label{04eqn:D2}\\
&\langle h, h\rangle+\langle g, g\rangle=1,
\label{04eqn:D3}\\
&\langle \bm{1}, h\rangle+\langle \bm{1}, g\rangle=0.
\label{04eqn:D4}
\end{align}
Taking the difference of \eqref{04eqn:D1} and
\eqref{04eqn:D2}, we obtain
\[
(\lambda+2)(h-g)=0.
\]
Here we do not need to check whether $\lambda=-2$ belongs to
$\Lambda(\Tilde{D})$ because of \eqref{4eqn:QEC formula}
and the fact that the QE constant is always $\ge-1$.

Thus we consider the case of $f=g$.
Equations \eqref{04eqn:D1}--\eqref{04eqn:D4} are reduced to
the following equations:
\begin{equation}\label{4eqn:D5}
(4D+4-2\lambda)h=\mu\bm{1},
\qquad
\langle h, h\rangle=\frac12\,,
\qquad
\langle \bm{1}, h\rangle=0.
\end{equation}
On the other hand, in computing $\mathrm{QEC}(G)$
we consider the set $\mathcal{S}(D)$ of
all $(f_1,\lambda_1,\mu_1)\in
C(V)\times \mathbb{R}\times\mathbb{R}$ satisfying
\begin{equation}\label{4eqn:D6}
(D-\lambda_1)f_1=\frac{\mu_1}{2}\,\bm{1}\,,
\qquad
\langle f_1, f_1\rangle=1\,,
\qquad
\langle \bm{1}, f_1\rangle=0.
\end{equation}
We see from \eqref{4eqn:D5} and \eqref{4eqn:D6}
that the correspondence
\[
h=\frac{1}{\sqrt2}\,f_1\,,
\qquad
\lambda=2\lambda_1+2,
\qquad
\mu=\sqrt{2}\,\mu_1,
\]
gives rise to a one-to-one correspondence between
$\mathcal{S}(D)$ and
$\{(f,f,\lambda,\mu)\in\mathcal{S}(\Tilde{D})\}$.
Consequently,
\begin{align*}
\mathrm{QEC}(\mathrm{Double\,}(G))
&=\max\{\lambda\,;\,
(g,h,\lambda,\mu)\in\mathcal{S}(\Tilde{D})\} \\
&=\max\{2\lambda_1+2\,;\,
(f_1,\lambda_1,\mu_1)\in\mathcal{S}(D)\} \\
&=2\mathrm{QEC}(G)+2,
\end{align*}
as desired.
\end{proof}

\begin{corollary}
Let $G=(V,E)$ be a finite connected graph with
$n=|V|\ge2$ vertices.
Then we have
\[
\mathrm{QEC}(\mathrm{Double\,}(G))\ge0.
\]
Moreover, $\mathrm{QEC}(\mathrm{Double\,}(G))=0$
if and only if $G=K_n$ is a complete graph.
\end{corollary}

\begin{proof}
Since $\mathrm{QEC}(G)\ge-1$ for any
finite connected graph with $n=|V|\ge2$ vertices,
we have
\[
\mathrm{QEC}(\mathrm{Double\,}(G))
=2\mathrm{QEC}(G)+2
\ge 2(-1)+2=0.
\]
The equality occurs if and only if $\mathrm{QEC}(G)=-1$,
that is, $G$ is a complete graph.
\end{proof}

\begin{example}\normalfont
Since $\mathrm{QEC}(K_n)=-1$ for $n\ge2$,
we have
\[
\mathrm{QEC}(\mathrm{Double\,}(K_n))=2\mathrm{QEC}(K_n)+2=0,
\qquad n\ge2.
\]
It is noted that $\mathrm{Double\,}(K_n)=K_{2n}\backslash nK_2$.
By elementary calculus we have
\begin{align*}
&\mathrm{QEC}(K_n\backslash K_2)=-\frac{2}{n}\,, \\
&\mathrm{QEC}(K_n\backslash rK_2)=0,
\qquad 2\le r\le\left[\frac{n}{2}\right],
\end{align*}
for details see \cite{Obata-Zakiyyah2018}.
\end{example}

\begin{remark}\normalfont
The distance spectrum of $\mathrm{Double\,}(G)$
is related to that of $G$ in a similar fashion
as in \eqref{4eqn:QEC(Double)},
for details see \cite{Indulal-Gutman2008}.
\end{remark}

\subsection{Lexicographic Product Graphs $G\triangleright K_2$}

The lexicographic product $G\triangleright K_2$ gives
another subgraph of the graph join $G+G$,
for the general definition of the lexicographic product,
see e.g., \cite{Hammack-etal}.

Here we introduce the lexicographic product $G\triangleright K_2$
in a slightly informal manner.
Let $G=(V,E)$ be a graph (not necessarily connected).
The lexicographic product $G\triangleright K_2$ is a graph
$(\Tilde{V},\Tilde{E})$ defined by
\begin{align*}
\Tilde{V}&=V\times\{0,1\}=\{(x,i)\,;\, x\in V, \, i\in\{0,1\}\},\\
\Tilde{E}&=\{
\{(x,i),(y,j)\}\,;\,(x,i)\in \Tilde{V},\,\,(y,j)\in \Tilde{V},\,\,
\text{$\{x,y\}\in E$ or $x=y$, $i\neq j$}  \}.
\end{align*}
In other words,
$G\triangleright K_2$ is obtained from $\mathrm{Double\,}(G)$
by adding new edges of the form $\{(x,0),(x,1)\}$, $x\in V$.
Obviously, $G\triangleright K_2$ is connected if and only if
so is $G$.

\begin{example}\normalfont
As is shown in Figure \ref{Fig:lexicographic graphs} we have
\[
K_2\triangleright K_2=K_4\,,
\qquad
K_3\triangleright K_2=K_6\,,
\qquad
P_3\triangleright K_2=K_6\backslash C_4\,.
\]
In fact, $K_n\triangleright K_2=K_{2n}$.
\begin{figure}[hbt]
\begin{center}
\includegraphics[width=60pt]{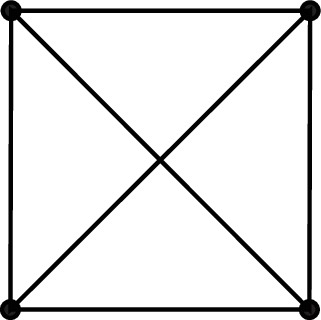}
\qquad\quad
\includegraphics[width=80pt]{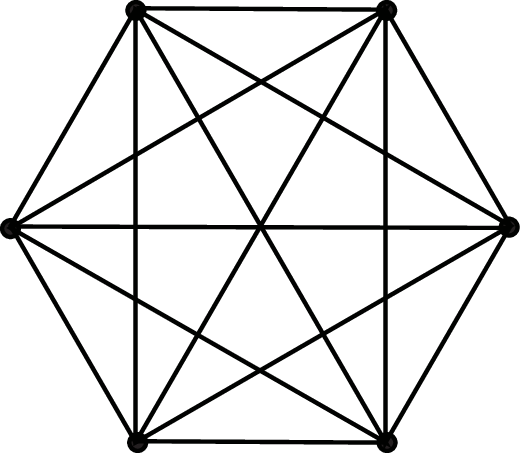}
\qquad\quad
\includegraphics[width=100pt]{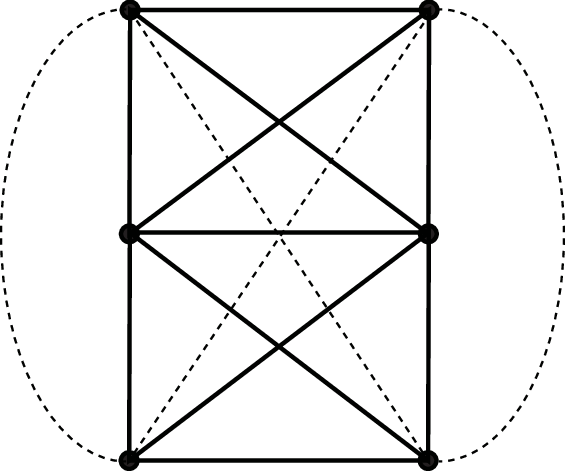}
\end{center}
\caption{$K_2\triangleright K_2$,\,\,
$K_3\triangleright K_2$ and $P_3\triangleright K_2$}
\label{Fig:lexicographic graphs}
\end{figure}
\end{example}

Let $G=(V,E)$ be a connected graph with $n=|V|\ge2$
and $D$ the distance matrix.
The distance matrix $\Tilde{D}$ of the lexicographic product
$G\triangleright K_2$ is written in the form:
\begin{equation}\label{05eqn:Tilde D of lexicographic product}
\Tilde{D}=
\begin{bmatrix}
D & D+I \\
D+I & D
\end{bmatrix},
\end{equation}
which is similar to the distance matrix of the double graph
$\mathrm{Double}(G)$.
Then after similar argument as in the case of
the double graphs we come to the following

\begin{theorem}\label{4thm:QEC(lexicographic)}
Let $G=(V,E)$ be a finite connected graph with
$n=|V|\ge2$ vertices.
Then we have
\begin{equation}\label{4eqn:QEC(lexicographic)}
\mathrm{QEC}(G\triangleright K_2)
=2\mathrm{QEC}(G)+1.
\end{equation}
\end{theorem}

\begin{remark}\normalfont
It follows from Theorem \ref{4thm:QEC(lexicographic)} that
$G\triangleright K_2$ is not of QE class
if and only if $\mathrm{QEC}(G)>-1/2$.
In this relation
characterizing finite connected graphs $G$ with
$\mathrm{QEC}(G)\le -1/2$ is an interesting question,
as is mentioned in Introduction.
\end{remark}

\begin{remark}\normalfont
The distance spectrum of $G\triangleright K_2$
is related to that of $G$ in a similar fashion
as in \eqref{4eqn:QEC(lexicographic)},
for details see \cite{Indulal-Gutman2008}.
\end{remark}

%%%%%%%%%%%%%%%%%%%%%%%%%%%%%%%%%%%%%%%%%%%%%%%%%%%%%%%%%%
\section{Strongly Regular Graphs}
\label{sec:Strongly Regular Graphs}

A \textit{strongly regular graph} with parameters $(n,r,e,f)$
is an $r$-regular graph on $n$ vertices
such that any two adjacent vertices have $e$ common neighbors
and any two non-adjacent vertices have $f$ common neighbors.
By definition the complete graph $K_n$ and
the empty graph $\bar{K}_n$ are \textit{not} strongly regular.
Strongly regular graphs, tracing back to Bose \cite{Bose1963},
have been extensively studied from various points of view,
for example see \cite{Brouwer1984,Brouwer-Haemers2012,Dragos2004}
as well as the large list \cite{BrouwerList}.

In this section we are interested in the QE constant of
a strongly regular graph.
Note that for a connected strongly regular graph $G$ we have
$n\ge4$, $2\le r\le n-2$, $f\ge1$ and $\mathrm{diam}(G)=2$.

\begin{proposition}\label{05prop:QEC(SRG)}
Let $G$ be a connected strongly regular graph
with parameters $(n,r,e,f)$.
Then we have
\[
\mathrm{QEC}(G)
=-2-\lambda_{\min}(G)
=-2+\frac{(f-e)+\sqrt{(f-e)^2+4(r-f)}}{2}\,,
\]
where $\lambda_{\min}(G)$ is the minimal eigenvalue of
the adjacency matrix of $G$.
\end{proposition}

\begin{proof}
The first equality follows from
Proposition \ref{02prop:QEC(regular and diam<=2}.
For the second equality let $A$ be the adjacency matrix of $G$.
An elementary observation leads
\[
(A-rI)(A^2+(f-e)A+(f-r)I)=0,
\]
from which we obtain
\[
\lambda_{\min}(G)=
\frac{-(f-e)-\sqrt{(f-e)^2+4(r-f)}}{2}\,,
\]
for details see e.g., \cite{Brouwer-Haemers2012,Cvetkovic1980}.
\end{proof}

\begin{theorem}\label{05thm:QEC(SRG with min-ev=-2)}
Let $G$ be a connected strongly regular graph.
Then $\mathrm{QEC}(G)\le0$ (resp. $\mathrm{QEC}(G)\ge0$)
if and only if $\lambda_{\min}(G)\ge-2$
(resp. $\lambda_{\min}(G)\le -2$).
\end{theorem}

\begin{proof}
A direct consequence from Proposition \ref{05prop:QEC(SRG)}.
\end{proof}

Graphs with smallest eigenvalue at least $-2$ have been
extensively studied after the famous classification
\cite{Cameron1976}.
Here we focus on the following significant result.

\begin{proposition}[\cite{Brouwer-Cohen-Neumaier1989}]
\label{05prop:BCN}
Let $G$ be a connected regular graph on $n\ge1$ vertices.
If $\lambda_{\min}(G)>-2$,
then either $G$ is a complete graph $K_n$ or
$G$ is a cycle $C_n$ with odd $n\ge3$.
\end{proposition}

\begin{theorem}\label{5-thm-1}
For any connected strong regular graph $\Gamma$ except $C_5$
we have
\begin{equation}\label{05eqn:in 5.9}
\mathrm{QEC}(\Gamma+K_m)\geq0,
\qquad
\mathrm{QEC}(\Gamma+\bar{K}_m)\geq0,
\qquad m\ge1.
\end{equation}
\end{theorem}

\begin{proof}
Let $\Gamma$ be a connected strong regular graph
and suppose that $\lambda_{\min}(\Gamma)>-2$.
It follows from Proposition \ref{05prop:BCN} that
either $\Gamma$ is a complete graph $K_n$ or
$\Gamma$ is a cycle $C_n$ with odd $n\ge3$.
Among them only $C_5$ is strongly regular.
Therefore,
for any connected strong regular graph $\Gamma$ except $C_5$
we have $\lambda_{\min}(\Gamma)\le -2$.
For such a graph $\Gamma$,
with the help of Corollary \ref{03cor:bound QEC(G_1+G_2)}
we see that
\[
\mathrm{QEC}(\Gamma+K_m)
\ge-2-\lambda_{\min}(\Gamma)\ge0,
\qquad m\ge1.
\]
Similarly, $\mathrm{QEC}(\Gamma+\bar{K}_m)\ge0$ is shown.
\end{proof}

\begin{example}[cycle $C_5$]
\normalfont
Already studied in Example \ref{exa-1}.
First note that
\[
\lambda_{\min}(C_5)=-\frac{1+\sqrt5}{2}>-2.
\]
We have
\begin{align*}
&\mathrm{QEC}(C_5+K_1)
=\mathrm{QEC}(C_5+K_2)
=-\dfrac{3-\sqrt5}{2}<0, \\[3pt]
&\mathrm{QEC}(C_5+K_m)=\dfrac{m-5}{m+5}\,,
\qquad m\ge3, \\[3pt]
&\mathrm{QEC}(C_5+\bar{K}_m)=\frac{6m-10}{m+5}\,,
\qquad m\ge2.
\end{align*}
\end{example}

For strongly regular graphs the following result is very interesting.
For definitions and properties of the graphs
in the statement, see e.g.,
\cite{Brouwer-Haemers2012,Cameron1976,Dragos2004}.

\begin{proposition}[\cite{Cheng-Greaves-Koolen2018}]
\label{05prop:CGK}
Let $\Gamma$ be a connected, coconnected,
strongly regular graph with
$\lambda_{\min}(\Gamma)\ge-2$.
Then $\Gamma$ is either a $5$-cycle $C_5$,
a triangular graph $T(n)$ with $n\ge5$,
an $(n\times n)$-grid with $n\ge3$,
the Petersen graph,
the Shrikhande graph,
the Clebsch graph,
the Schl\"afli graph,
or one of the three Chang graphs.
\end{proposition}

Below we list the QE constants of the joins of
a graph mentioned in Proposition \ref{05prop:CGK}
with $K_n$ and $\bar{K}_n$.
The calculation is routine application of
Theorem \ref{03thm:QEC(G_1+G_2)}.

\begin{example}[triangular graphs]\label{ex:triangular graphs}
\normalfont
The line graph of the complete graph $K_n$ is
called a \textit{triangular graph}
and denoted by $T(n)$ where $n\ge2$.
For $n\ge4$ the triangular graph $T(n)$ is strongly regular with
parameters:
\[
\bigg(\frac{n(n-1)}{2}, 2(n-2), n-2,4\bigg).
\]
Moreover,
\begin{align*}
&\mathrm{ev}(T(n))=\{2(n-2)^1,\,\,
(n-4)^{n-1},\,\, (-2)^{n(n-3)/2}\},
\\
&\mathrm{QEC}(T(n))=0.
\end{align*}
For graph joins we have
\begin{align*}
&\mathrm{QEC}(T(n)+K_m)=
\max\bigg\{0,\,
\frac{(n-1)(nm-n-4m)}{n(n-1)+2m}\bigg\},\\[3pt]
&\mathrm{QEC}(T(n)+\bar{K}_m)=
\max\bigg\{0,\,
\frac{2(n-1)(nm-n-2m)}{n(n-1)+2m}\bigg\},
\end{align*}
where $n\ge4$ and $m\ge1$.
Note that $T(4)\cong K_{2,2,2}$ and hence
$\overline{T(4)}\cong 3K_2$ is not connected.
\end{example}

\begin{example}[grids]\label{ex:grids}
\normalfont
For $n\ge2$ the Cartesian product $K_n\times K_n$
is called an \textit{$(n\times n)$-grid}.
The $(n\times n)$-grid is strongly regular with
parameters $(n^2,2(n-1),n-2,2)$.
Moreover,
\begin{align*}
&\mathrm{ev}(K_n\times K_n)
=\{2(n-1)^{1},(n-2)^{2n-2}, \,\, (-2)^{n^2-2n+1}\},\\
&\mathrm{QEC}(K_n\times K_n)=0.
\end{align*}
For graph joins we have
\begin{align*}
&\mathrm{QEC}((K_n\times K_n)+K_m)=
\max\bigg\{0,\,
\frac{n(nm-n-2m)}{n^2+m}\bigg\},
\quad n\ge2,\,\, m\ge1, \\[3pt]
&\mathrm{QEC}((K_n\times K_n)+\bar{K}_m)
=\frac{2n(nm-n-m)}{n^2+m}\,,
\quad n\ge2,\,\, m\ge2.
\end{align*}
Note that $K_2\times K_2\cong C_4\cong K_{2,2}$ and hence
$\overline{K_2\times K_2}\cong 2K_2$ is not connected.
\end{example}

\begin{example}[Petersen graph]\label{ex:Petersen graph}
\normalfont
The Petersen graph $\Gamma_1$ is
the unique strongly regular graph with parameters $(10,3,0,1)$.
We have
\[
\mathrm{ev}(\Gamma_1)=\{3^1, 1^5, (-2)^4\},
\qquad
\mathrm{QEC}(\Gamma_1)=0.
\]
For graph joins we have
\begin{align*}
&\mathrm{QEC}(\Gamma_1+K_1)=\mathrm{QEC}(\Gamma_1+K_2)=0, \\[3pt]
&\mathrm{QEC}(\Gamma_1+K_m)=\frac{5m-10}{m+10}>0, \quad m\geq3, \\[3pt]
&\mathrm{QEC}(\Gamma_1+\bar{K}_m)=\frac{15m-20}{m+10}>0,
\quad m\geq2.
\end{align*}
\end{example}

\begin{example}[Shrikhande graph]\normalfont
There are two strongly regular graphs with
parameters $(16, 6, 2, 2)$.
One is $(4\times 4)$-grid and the other is
the Shrikhande graph $\Gamma_2$.
We have
\[
\mathrm{ev}(\Gamma_2)=\{6^1, 2^6, (-2)^9\},
\qquad
\mathrm{QEC}(\Gamma_2)=0.
\]
For graph joins we have
\begin{align*}
&\mathrm{QEC}(\Gamma_2+K_1)=\mathrm{QEC}(\Gamma_2+K_2)=0, \\[3pt]
&\mathrm{QEC}(\Gamma_2+K_m)=\frac{8m-16}{m+16}>0, \quad m\geq3,
\\[3pt]
&\mathrm{QEC}(\Gamma_2+\bar{K}_m)=\frac{24m-32}{m+16}>0,
\quad m\geq2.
\end{align*}
\end{example}

\begin{example}[Clebsch graph]\normalfont
The ($10$-regular) Clebsch graph $\Gamma_3$ is
the unique strongly regular graph with parameters $(16,10,6,6)$.
We have
\[
\mathrm{ev}(\Gamma_3)=\{10^1, 2^5, (-2)^{10}\},
\qquad
\mathrm{QEC}(\Gamma_3)=0.
\]
For graph joins we have
\begin{align*}
&\mathrm{QEC}(\Gamma_3+K_m)=0, \quad 1\le m\le 4, \\[3pt]
&\mathrm{QEC}(\Gamma_3+K_m)=\frac{4m-16}{m+16}>0, \quad m\geq5,\\[3pt]
&\mathrm{QEC}(\Gamma_3+\bar{K}_m)=\frac{20m-32}{m+16}>0,
\quad m\geq2.
\end{align*}
\end{example}

\begin{example}[Schl\"{a}fli graph]\normalfont
The Schl\"{a}fli graph $\Gamma_4$ is
the unique strongly regular graph with parameters $(27,16,10,8)$.
We have
\[
\mathrm{ev}(\Gamma_4)=\{16^1, 4^6, (-2)^{20}\},
\qquad
\mathrm{QEC}(\Gamma_4)=0.
\]
For graph joins we have
\begin{align*}
&\mathrm{QEC}(\Gamma_4+K_1)=\mathrm{QEC}(\Gamma_4+K_2)
=\mathrm{QEC}(\Gamma_4+K_3)=0, \\[3pt]
&\mathrm{QEC}(\Gamma_4+K_m)=\frac{9m-27}{m+27}>0,
\quad m\geq4, \\[3pt]
&\mathrm{QEC}(\Gamma_4+\bar{K}_m)=\frac{36m-54}{m+27}>0,
\quad m\geq2.
\end{align*}
\end{example}

\begin{example}[Chang graphs]\label{ex:Chang graphs}
\normalfont
There are four strongly regular graphs with parameters $(28,12,6,4)$.
One is the triangular graph $T(8)$, i.e.,
the line graph of $K_8$.
Each of the other three is called a Chang graph.
For a Chang graph $\Gamma_5$ we have
\[
\mathrm{ev}(\Gamma_5)=\{12^1, 4^7, (-2)^{20}\},
\qquad
\mathrm{QEC}(\Gamma_5)=0.
\]
For graph joins we have
\begin{align*}
&\mathrm{QEC}(\Gamma_5+K_1)=\mathrm{QEC}(\Gamma_5+K_2)=0, \\[3pt]
&\mathrm{QEC}(\Gamma_5+K_m)=\frac{14m-28}{m+28}>0,
\quad m\geq3, \\[3pt]
&\mathrm{QEC}(\Gamma_5+\bar{K}_m)=\frac{42m-56}{m+28}>0,
\quad m\geq2.
\end{align*}
\end{example}

Every strongly regular graph in the above Examples
\ref{ex:triangular graphs}--\ref{ex:Chang graphs}
has the minimal eigenvalue $-2$ and hence the zero QE constant.
We add some more examples with positive QE constants.

\begin{example}[Hoffman-Singleton graph]
\normalfont
The Hoffman-Singleton graph $\Gamma_6$ is
the unique strongly regular graph with
parameters $(50,7, 0,1)$.
We have
\[
\mathrm{ev}(\Gamma_6)=\{7^1, 2^{28}, (-3)^{21}\},
\qquad
\mathrm{QEC}(\Gamma_6)=1.
\]
For graph joins we have
\begin{align*}
&\mathrm{QEC}(\Gamma_6+{K}_m)=1, \quad m=1,2, \\[3pt]
&\mathrm{QEC}(\Gamma_6+{K}_m)=\frac{41m-50}{m+50}>0,
\quad m\geq3, \\[3pt]
&\mathrm{QEC}(\Gamma_6+\bar{K}_m)=\frac{91m-100}{m+50}>0,
\quad m\geq2.
\end{align*}
\end{example}

\begin{example}[Higman-Sims graph]
\normalfont
The Higman-Sims graph $\Gamma_{7}$ is the
unique strongly regular graph  with parameters $(100,22,0,6)$.
We have
\[
\mathrm{ev}(\Gamma_7)=\{22^1, 2^{77}, (-8)^{22}\},
\qquad
\mathrm{QEC}(\Gamma_7)=6.
\]
For graph joins we have
\begin{align*}
&\mathrm{QEC}(\Gamma_7+K_m)=6, \quad 1\leq m\leq9, \\[3pt]
&\mathrm{QEC}(\Gamma_7+K_m)=\frac{76m-100}{m+100}>0, \quad m\geq10,
\\[3pt]
& \mathrm{QEC}(\Gamma_7+\bar{K}_m)=6, \quad 1\le m\le 4, \\[3pt]
&\mathrm{QEC}(\Gamma_7+\bar{K}_m)
=\frac{176m-200}{m+100}>0,
\quad m\geq5.
\end{align*}
\end{example}

\begin{example}[Suzuki graph]
\normalfont
The Suzuki graph $\Gamma_{8}$ is a rank $3$
strongly regular graph with parameters $(1782,416,100,96)$.
We have
\[
\mathrm{ev}(\Gamma_{8})=\{416^1, 20^{780}, (-16)^{1001}\},
\qquad
\mathrm{QEC}(\Gamma_{8})=14.
\]
For graph joins we have
\begin{align*}
&\mathrm{QEC}(\Gamma_{8}+K_m)=14, \quad 1\leq m\leq19, \\[3pt]
&\mathrm{QEC}(\Gamma_{8}+K_m)=\frac{1364m-1782}{m+1782}>0,
\quad m\geq20.\\[3pt]
&\mathrm{QEC}(\Gamma_{8}+\bar{K}_m)=14,  \quad 1\leq m\leq9, \\[3pt]
&\mathrm{QEC}(\Gamma_{8}+\bar{K}_m)=\frac{3146m-3564}{m+1782}>0,
\quad m\geq10.
\end{align*}
\end{example}

\section*{Declarations}

\textbf{Conflicts of interest} No competing interests.\\

\textbf{Availability of Data and Material} Not applicable.


\begin{thebibliography}{ll}

\bibitem{Alfakih2018}

A. Y. Alfakih:
``Euclidean Distance Matrices and Their Applications in Rigidity Theory,"
Springer, Cham, 2018.

\bibitem{Aouchiche-Hansen2014}
M. Aouchiche and P. Hansen:
\textit{Distance spectra of graphs: a survey},
Linear Algebra Appl. {\bfseries 458} (2014), 301--386.

\bibitem{Balaji-Bapat2007}
R. Balaji and R. B. Bapat:
\textit{On Euclidean distance matrices},
Linear Algebra Appl. {\bfseries 424} (2007), 108--117.

\bibitem{Bapat2010}
R. B. Bapat: ``Graphs and Matrices,"
Springer, Hindustan Book Agency, New Delhi, 2010.

\bibitem{BO-2019}
E. T. Baskoro and N. Obata:
\textit{Determining finite connected graphs along the quadratic
embedding constants of paths},
Electron. J. Graph Theory Appl. {\bfseries 9} (2021), 539--560.

\bibitem{BiyikogluLeydold2007}
T. B\i y\i ko\v{g}lu, J. Leydold and P. F. Stadler:
``Laplacian Eigenvectors of Graphs.
Perron-Frobenius and Faber-Krahn Type Theorems,"
Lect. Notes Math. Vol.~1915. Springer, Berlin, 2007.

\bibitem{Bose1963}
R. C. Bose:
\textit{Strongly regular graphs, partial geometries and
partially balanced designs},
Pacific J. Math. {\bfseries13} (1963), 389--419.

\bibitem{Bozejko89}
M. Bo\.zejko:
\textit{Positive-definite kernels, length functions on groups
and noncommutative von Neumann inequality},
Studia Math. \textbf{95} (1989), 107--118.


\bibitem{BrouwerList}
A. E. Brouwer:
https://www.win.tue.nl/\~{}aeb/graphs/srg/srgtab.html

\bibitem{Brouwer-Cohen-Neumaier1989}
A. E. Brouwer, A. M. Cohen and A. Neumaier:
``Distance-Regular Graphs,'' Springer, Berlin Heidelberg, 1989.

\bibitem{Brouwer-Haemers2012}
A. E. Brouwer and W. H. Haemers:
``Spectra of Graphs," Springer, New York, 2012.

\bibitem{Brouwer1984}
A. E. Brouwer and  J. H. van Lint:
\textit{Strongly regular graphs and partial geometries},
in ``Enumeration and Design (D. M. Jackson and S. A. Vanstone, Eds.),''
pp.~85--122, Academic Press, Toronto, 1984.

\bibitem{Cameron1976}
P. J. Cameron, J. M. Goethals, J. J. Seidel and E. E. Shult:
\textit{Line graphs, root systems and elliptic geometry},
J. Algebra {\bfseries43} (1976), 305--327.

\bibitem{Cheng-Greaves-Koolen2018}
X. M. Cheng, G. R. W. Greaves and J. H. Koolen:
\textit{Graphs with three eigenvalues
and second largest eigenvalue at most $1$},
J. Combinatorial Theory, Series B {\bfseries129} (2018), 55--78.

\bibitem{Cvetkovic1978}
D. M. Cvetkovi\'{c}:
\textit{The main part of the spectrum, divisors and
switching of graphs},
Publ. Inst. Math. (Beograd) (N.S.) 23(37) (1978), 31--38.

\bibitem{Cvetkovic1980}
D. M. Cvetkovic, M. Doob and H. Sachs:
\textit{Spectra of Graphs-Theory and Applications},
Academic Press, New York, 1980.




\bibitem{Dragos2004}
D. M. Cvetkovi\'{c}, P. Rowlinson and S. Simi\'{c}:
\textit{Spectral generalizations of line graphs:
on graphs with least negative eigenvalue $-2$},
London Mathematical Society Lecture Note Series 314, 2004.

\bibitem{Deza-Laurent1997}
M. M. Deza and M. Laurent:
``Geometry of Cuts and Metrics,''
Springer-Verlag, Berlin, 1997.


\bibitem{Graham-Winkler1985}
R. L. Graham and P. M. Winkler:
\textit{On isometric embedding of graphs},
Trans. Amer. Math. Soc. {\bfseries 288} (1985), 527--536.


\bibitem{Hammack-etal}
R. Hammack, W. Imrich and S. Klav\u{z}ar:
``Handbook of Product Graphs (2nd Ed.),"
CRC Press, Boca Raton, FL, 2011.

\bibitem{Hora-Obata2007}
A. Hora and N. Obata:
``Quantum Probability and Spectral Analysis of Graphs,"
Springer, Berlin, 2007.

\bibitem{Indulal-Gutman2008}
G. Indulal and I. Gutman:
\textit{On the distance spectra of some graphs},
Mathematical Communications. {\bfseries 13} (2008), 123--131.

\bibitem{Irawan-Sugeng2021}
W. Irawan and K. A. Sugeng:
\textit{Quadratic embedding constants of hairy cycle graphs},
Journal of Physics: Conference Series {\bfseries 1722} (2021), 012046.


\bibitem{Jaklic-Modic2010}
G. Jakli\v{c} and J. Modic:
\textit{On properties of cell matrices}.
Appl. Math. Comput. {\bfseries 216} (2010), 2016--2023.

\bibitem{Jaklic-Modic2013}
G. Jakli\v{c} and J. Modic:
\textit{On Euclidean distance matrices of graphs},
Electron. J. Linear Algebra {\bfseries 26} (2013), 574--589.

\bibitem{Jaklic-Modic2014}
G. Jakli\v{c} and J. Modic:
\textit{Euclidean graph distance matrices of generalizations
of the star graph},
Appl. Math. Comput. {\bfseries 230} (2014), 650--663.

\bibitem{Koolen-Shpectorov1994}
J. H. Koolen and S. V. Shpectorov:
\textit{Distance-regular graphs the distance matrix of
which has only one positive eigenvalue},
European J. Combin. {\bfseries 15} (1994), 269--275.

\bibitem{Liberti-etal2014}
L. Liberti, G. Lavor, N. Maculan and A. Mucherino:
\textit{Euclidean distance geometry and applications},
SIAM Rev. {\bfseries 56} (2014), 3--69.

\bibitem{Lin2015}
H.  Lin:
\textit{Proof of a conjecture involving the second largest
$D$-eigenvalue and the number of triangles},
Linear Algebra Appl. {\bfseries472} (2015), 48--53.

\bibitem{Lin2013}
H.  Lin, Y. Hong, J. F. Wang and J. L. Shu:
\textit{On the distance spectrum of graphs},
Linear Algebra Appl. {\bfseries 439} (6) (2013), 1662--1669.

\bibitem{Liu-Xue-Guo2015}
R. Liu, J. Xue and L. Guo:
\textit{On the second largest distance eigenvalue of a graph},
 Linear Multilinear Algebra. {\bfseries 65} (5) (2016),  1011--1021.

\bibitem{Maehata2013}
H. Maehata:
\textit{Euclidean embeddings of finite metric spaces},
Discrete Math. {\bfseries 313} (2013), 2848--2856.


\bibitem{M-2022}
W. M\l otkowski:
\textit{Quadratic embedding constants of path graphs},
Linear Algebra Appl. {\bfseries 644} (2022),  95--107.


\bibitem{MO-2018}
W. M\l otkowski and N. Obata:
\textit{On quadratic embedding constants of star product graphs},
Hokkaido Math. J. {\bfseries 49} (2020), 129--163.

\bibitem{Munarini-etal2008}
E. Munarini, C. P. Cippo, A. Scagliola and N. Z. Salvi:
\textit{Double graphs},
Discrete Mathematics {\bfseries 308} (2008), 242--254.

\bibitem{Obata2017}
N. Obata:
\textit{Quadratic embedding constants of wheel graphs},
Interdiscip. Inform. Sci. {\bfseries 23} (2017), 171--174.

\bibitem{Obata-Zakiyyah2018}
N. Obata and A. Y. Zakiyyah:
\textit{Distance matrices and quadratic embedding of graphs},
Electronic J. Graph Theory Appl. {\bfseries 6} (2018), 37--60.

\bibitem{Purwaningsih-Sugeng2021}
M. Purwaningsih and K. A. Sugeng:
\textit{Quadratic embedding constants of squid graph and kite graph},
J. Phys.: Conf. Ser. {\bfseries 1722} (2021), 012047.


\bibitem{Schoenberg1935}
I. J. Schoenberg:
\textit{Remarks to Maurice Fr\'echet's article
``Sur la d\'efinition axiomatique d'une
classe d'espace distanci\'s vectoriellement applicable sur l'espace de Hilbert'',}
Ann. of Math. {\bfseries 36} (1935), 724--732.

\bibitem{Schoenberg1938}
I. J. Schoenberg:
\textit{Metric spaces and positive definite functions},
Trans. Amer. Math. Soc. {\bfseries 44} (1938), 522--536.

\bibitem{VanMieghem2011}
P. van Mieghem:
``Graph Spectra for Complex Networks,"
Cambridge University Press, Cambridge, 2011.


\bibitem{Young-Householder1938}
G. Young and A. S. Householder:
\textit{Discussion of a set of points in terms of their mutual distances},
Psychometrika {\bfseries 3} (1938), 1--2.

\end{thebibliography}
\end{document}